\documentclass[final,2p,times,11pt]{elsarticle}
\usepackage{amsmath,amssymb,amsthm, epsfig}
\usepackage{hyperref}
\usepackage{amsmath}
\usepackage{amssymb}
\usepackage{color}
\usepackage{ulem}
\usepackage{epsfig}
\usepackage[mathscr]{eucal}

\usepackage{mathptmx} 
\usepackage[scaled=0.90]{helvet} 
\usepackage{courier} 
\normalfont
\usepackage[T1]{fontenc}

\newlength{\hchng}
\newlength{\vchng}
\def \osc {\mathrm{osc}}

\newcommand{\defeq}{\mathrel{\mathop:}=}

\newtheorem{theorem}{Theorem}[section]

\newtheorem{lemma}[theorem]{Lemma}

\newtheorem{proposition}[theorem]{Proposition}

\newtheorem{corollary}[theorem]{Corollary}

\theoremstyle{definition}

\theoremstyle{remark}

\newtheorem{remark}[theorem]{Remark}

\numberwithin{equation}{section}

\newcommand{\intav}[1]{\mathchoice {\mathop{\vrule width 6pt height 3 pt depth  -2.5pt
\kern -8pt \intop}\nolimits_{\kern -6pt#1}} {\mathop{\vrule width
5pt height 3  pt depth -2.6pt \kern -6pt \intop}\nolimits_{#1}}
{\mathop{\vrule width 5pt height 3 pt depth -2.6pt \kern -6pt
\intop}\nolimits_{#1}} {\mathop{\vrule width 5pt height 3 pt depth
-2.6pt \kern -6pt \intop}\nolimits_{#1}}}

\begin{document}

\begin{frontmatter}

\title{Optimal $C^{1,\alpha}$ regularity  for degenerate fully nonlinear elliptic equations with Neumann boundary condition}

\author{Gleydson C. Ricarte\footnote{Universidade Federal do Cear\'a, Department of Mathematics, Fortaleza, CE-Brazil 60455-760, \texttt{ricarte@mat.ufc.br} }  }
\medskip


\begin{abstract}
In the present paper, we study sharp $C^{1,\alpha}$ regularity results with boundary Neumann condition for viscosity solutions for a class of degenerate fully non-linear elliptic equations.

\end{abstract}

\begin{keyword}
    Regularity theory, optimal {\it a priori} estimates, fully nonlinear elliptic equations.\\
    \vspace{0.5 cm}

\end{keyword}

\end{frontmatter}

\section{Introduction}
The quest for obtaining sharp, optimal regularity results is one of the most exciting current trends in the study of nonlinear pdes. In this paper, we study sharp $C^{1,\alpha}$ regularity estimates of viscosity solutions for
\begin{equation} \label{E1}
	|\nabla u|^{\gamma} F(D^2 u) =f,
\end{equation}
with Neumann boundary condition $\nabla u \cdot \nu = g$ when $\Omega \subset \mathbb{R}^n$ is a bounded domain with  $ \Omega \in C^{2}$, $\nu$ denotes the inner normal of $\Omega$,  $f$ is a function defined in $\Omega$, $g$ is a function defined on $\partial \Omega$ and $\gamma>0$, $F$ is uniformly elliptic, $F(0)=0$.  Notice that solutions $u$ of \eqref{E1} cannot be more regular than $C^{1,\alpha}$. More precisely for $0 < \alpha <1$, the function 
$$
	u(x) = |x|^{1+\alpha}
$$
(as mentioned in \cite{ART2,Silv1}) satisfies 
$$
	|\nabla u|^{\gamma} \Delta u =C |x|^{(1+\alpha)(\gamma+1)-(\gamma +2)}
$$
where $C=(1+\alpha)^{1+\gamma}(n+\alpha-1)$.  The HRS is in $L^{\infty}$, if we choose $\alpha=\frac{1}{1+\gamma}$.  This example shows that the best regularity that one can expect for solutions to \eqref{E1} is $C^{1,\frac{1}{1+\gamma}}$.  The specificity of these equations is that they  are not uniformly elliptic; they are either singular or degenerate (in a way to be made precise).  {\it Singular/degenerate fully non-linear elliptic equations } of the type \eqref{E1} makes part of a class of non-linear elliptic equations studied in a series of papers by Birindelli and Demengel, starting with \cite{BD} (for singular case) and \cite{BD1,BD2,BD3} for the degenerate case.    Interior regularity properties of viscosity solutions of \eqref{E1} have been studied since a long time, starting with the seminal paper of C. Imbert and L. Silvestre  \cite{Silv1},  which contains interior $C^{1,\alpha}$ estimates for \eqref{E1} when $f \in L^{\infty}(B_1)$. Very recently, optimal regularity was proved in \cite{ART2}, with the optimal H\"{o}lder's coefficient $\alpha = \frac{1}{1+\gamma}$.  

For Dirichlet boundary data, the regularity up to the boundary is fairly well understood. To know,  Birindelli and Demengel \cite{BD4} proved $C^{1,\alpha}(\overline{\Omega})$ estimates in the presence of a regular boundary datum.  However, for the Neumann problem, there are still not many results. Regularity properties of viscosity solutions of fully nonlinear elliptic equations $F(D^2 u)=f$ with Neumann boundary condition have been studied the seminal paper Milakis, E and Silvestre L. \cite{Silv},  which contains in particular $C^{1,\alpha}$  estimates  viscosity solutions of 
$$
\left\{
\begin{array}{rcl}
F(D^2u) &=& f  \quad \mbox{in} \, \Omega\\
 \nabla u \cdot \nu &=& g  \quad \mbox{in} \,\,  \partial \Omega,
\end{array}
\right.
$$
when $f \in L^p(\Omega)$, $p>n$ and $g \in C^{\beta}(\partial \Omega)$.  Inspired by that breakthrough, we wanted to complete the work, showing that a similar result holds at the flat boundary (see Theorems \ref{P1}). 

Very recently, the $C^{1,\alpha}$  estimate for degenerate fully nonlinear equations of the type 
\begin{equation} \label{BB1}
\left\{
\begin{array}{rcl}
|\nabla u|^{\gamma}F(D^2u) &=& f  \quad \mbox{in} \, \Omega\\
 \nabla u \cdot \nu &=& g  \quad \mbox{in} \,\,  \partial \Omega,
\end{array}
\right.
\end{equation}
has been solved in \cite{BB} (see too \cite{Patrizi}).  In this paper we will develop the optimal regularity theory for \eqref{BB1}. Precisely, we will apply the technique presented in \cite{ART2} to prove that viscosity solution to \eqref{BB1} are $C^{1,\alpha}$, where $\alpha = \textrm{min} \{\alpha^{-}_0, \beta, \frac{1}{1+\gamma}\}$ and $\alpha_0 \in (0,1)$ is the optimal exponent of regularity theory for homogeneous equations (see \eqref{Hol1}).

The paper is organized as follows. In Section \ref{SC0} we specify the notation to be used in the paper and main results. In Section \ref{SC2} we will prove $C^{1,\alpha}$ regularity.  Getting $C^{1,\alpha}$ estimates consists in proving that the graph of the function $u$ can be approximated by planes with an error bounded by $Cr^{1+\alpha}$ in semi-balls of radius $r$.  The proof is based on an iterative argument, in which we show that the graph of $u$ gets flatter in smaller semi-balls.  When we consider the problem \eqref{E1} the principal difficulty lies in the following fact: if $\ell(x)=a+\vec p \cdot x$ is an affine function and $u$ is a viscosity solution for the problem \eqref{E1}, we can not conclude that $u + \ell$ is a viscosity solution for the problem \eqref{E1}.  In \cite{Silv}, this fact is important because it allows us to apply regularity theory for $v=u + \ell$ which is crucial to reach a {\it improvement of flatness} for the problem
\begin{equation}
	|\nabla u + \vec q|^{\gamma} F(D^2 u)=f \quad \textrm{in} \quad B_1, \label{q}
\end{equation}
where $\vec q \in \mathbb{R}^n$.  In Sect. \ref{Four} we provide a few implications the sharp estimates from Sect. 3 have
towards the solvability of some well known open problems in the elliptic regularity theory

\section{Notations and the statement of the main result}\label{SC0}

In this section, we will present notations and main assumptions which we will work throughout this article. Furthermore, we will also collect some preliminary results for future references. For the reader's convenience we recall the definition of viscosity solutions of fully nonlinear elliptic equations and provide a brief collection of basic results related to this notion. We always assume that $F$ is $(\lambda,\Lambda)$-elliptic, i.e., there exists constants $0 \le \lambda \le \Lambda < \infty$ such that
$$
	\lambda \|N\| \le F(M+N,x)-F(M,x) \le \Lambda \|N\|
$$
holds for $M,N \in \textrm{Sym}(n)$, $N \ge 0$ and $x \in \Omega$, where $\|M\| \colon= \sup_{|x|=1} |Mx|$. Also $\nabla u$ will denote the total gradient of $u$.   We will introduce the well-known \textit{Pucci's extremal operators}: Let $0 < \lambda \le \Lambda$ be given constants. For $M \in \mathcal{S}(n)$ we define:
\begin{equation}
 	\mathscr{P}^{+}_{\lambda,\Lambda}(M) \defeq \Lambda \sum_{e_i >0} e_i + \lambda \sum_{e_i <0} e_i \quad \text{and} \quad \mathscr{P}^{+}_{\lambda,\Lambda}(M) \defeq \lambda \sum_{e_i >0} e_i + \Lambda \sum_{e_i <0} e_i,
 \end{equation}
 where $e_i = e_i(M)$ denote the eigenvalues of $M$. We say that $u$ is  a viscosity subsolution (supersolution) of \eqref{E1} if for any $\varphi \in C^2(\Omega \cup \partial \Omega)$ touching $u$ by above (below) at $x_0 $ in $\Omega \cup \partial \Omega$, we have that
$$
	|\nabla \varphi(x_0)|^{\gamma} F(D^2 \varphi(x_0)) \ge \,\, (\le) \, f(x_0) \quad \textrm{if} \quad x_0 \in \Omega
$$
and
$$
	\nabla \varphi(x_0) \cdot \nu   \ge \,\, (\le) \,\, g(x_0) \quad \textrm{if} \quad x_0 \in \partial \Omega,
$$
 If $u$ is both subsolution and supersolution, we call it a viscosity solution.  We will denote by $C^{0,\alpha}(x_0)$ (or $C^{1,\alpha}(X_0)$) the class of functions $v$ which are H\"{o}lder continuous (or which have H\"{o}lder continuous first order derivatives with H\'{o}lder exponent $\alpha \in (0,1)$). By $\|\cdot\|_{C^{\alpha}}$ we mean the maximum of $L^{\infty}$-norm and the $\alpha$-H\"{o}lder semi-norm. Also by $\|\cdot\|_{C^{1,\alpha}}$ we mean the maximum of $L^{\infty}$-norm and the $\|\cdot\|_{C^{\alpha}}$-norm of the gradient.


 Now, we state our main result.




\begin{theorem}[$C^{1,\alpha}$ regularity]\label{P1}
Let $\Omega \in C^2$, $0 \in \partial \Omega$ and  $u$ be a viscosity solution to 
 \begin{equation} \label{Flat-1}
\left\{
\begin{array}{rcl}
| \nabla u|^{\gamma} F(D^2u) &=& f  \quad \mbox{in} \, \Omega \cap B_1(0)\\
 \nabla u \cdot \nu &=& g  \quad \mbox{in} \,\, \partial \Omega \cap B_1(0)
\end{array}
\right.
\end{equation}
where $f \in C(\overline{\Omega})$ and $g \in C^{\beta}(\partial \Omega)$ for some $\beta \in (0,1)$. Then  $u \in C^{1,\alpha}(\overline{\Omega \cap B_{1/2}(0)})$ where $\alpha = \textrm{min} \left\{\alpha^{-}_0, \beta, \frac{1}{1+\gamma}\right\}$ and $\alpha_0$  is the optimal H\"{o}lder exponent for solutions to constant coefficient, homogeneous equation . Moreover, we have the estimate
\begin{equation} \label{flatq}
	\|u\|_{C^{1,\alpha}(\overline{\Omega \cap B_{1/2}})} \le C \left(\|u\|_{C(\overline{\Omega})} + \|g\|_{C^{\beta}(\partial \Omega)} + \|f\|_{L^{\infty}(\overline{\Omega})}\right)
\end{equation}
for a universal constant $C$ 
\end{theorem}

\subsection{Reduction to flat boundary conditions:} Since $\Omega$ is $C^2$, we can flatten the boundary using coordinates which employs the distance function to the boundary $\partial \Omega$. See for instance Lemma 14.16 in \cite{007} or the Appendix in \cite{008}. We crucially note that such coordinates preserve the Neumann boundary conditions unlike standard flattening which changes Neumann conditions to oblique derivative conditions in general. Consequently, without loss of generality, we may consider the following flat boundary value problem 
\begin{equation} \label{Neura}
 \left\{
\begin{array}{rcl}
 \langle \mathcal{A}(x) \nabla u, \nabla u\rangle^{\gamma/2}G(x,   \nabla u,D^2u) &=& f,  \quad \mbox{in} \, \,\,\, B^+_1\\
 \partial_n u &=& g  \quad \,\, \, \, \mbox{in} \,\, \Upsilon.
\end{array}
\right.
 \end{equation}
where $B^+_{r}=B_r(0) \cap \{x_n>0\}$ and $\Upsilon = B_r(0) \cap \{x_n=0\}$, $\mathcal{A}$ is a uniformly elliptic positive definite matrix with Lipschitz coefficients, $G: \textrm{Sym}(n) \times \mathbb{R}^n \times \mathbb{R}^n \rightarrow \mathbb{R}$ satisfies the following structural conditions:
\begin{enumerate}
\item[C1)] $F$ is degenerate elliptic, that is, for all $M,N \in \textrm{Sym}(n)$, $N \ge 0$ and $(p,x) \in \mathbb{R}^n \times \mathbb{R}^n$,
$$
	G(M+N,p,x) \ge G(M,p,x).
$$
\item[C2)] $G(0,p,x) =0$ for all $(p,x) \in \mathbb{R}^n \times \mathbb{R}^n$.
\item[C3)] $G$ is uniformly elliptic is a small neighborhood of $0$, i.e., there is a $\mu >0$ such that
$$
	\lambda \|N\| \le G(M+N,p,x) - G(M,p,x) \le \Lambda \|N\|,
$$
for some $0 < \lambda \le \Lambda$, $p \in B_{\mu}$, $M,N \in \textrm{Sym}(n)$ and $N \ge 0$ and $x \in \mathbb{R}^n$.
\end{enumerate} 
 Thus, the techniques in this paper can be modified to yield $C^{1,\alpha}$ regularity results for Neumann boundary problems of the type \eqref{Neura}.  Thus, for simplicity and clarity of the arguments we focus on the specific form $|\nabla u|^{\gamma}F(D^2 u)=f$ in $B^+_1$ with $\partial_n u =g$ in $\Upsilon$.

\subsection{Some important lemmas}

We introduce the following result.

 \begin{lemma}[See \cite{BB}]\label{LLL2}
 Let $\vec q \in \mathbb{R}^n$. Assume $w$ solves  
\begin{equation} \label{Plus}
	 |\nabla w + \vec q|^{\gamma}F(D^2u)  = f \quad \textrm{in} \quad B^+_1
\end{equation}
 with $|w| \le 1$, $\|f\|_{L^{\infty}(B^+_1)} \le 1$ and with Neumann boundary condition $ w_{x_n}  = g$ in $\Upsilon$ satisfying $\|g\|_{B^{\beta}(\Upsilon)} \le 1$. Then $w$ is equicontinuous up to boundary.
 \end{lemma}

We next use the available compactness to derive a mechanism linking solutions of the inhomogeneous pde and solutions of the homogeneous equation
\begin{corollary}\label{A1}
 Let $(\vec q_k)$ is a sequence in $\mathbb{R}^n$ and  $(f_k)_{k},(g_k)_k$ are sequences of continuous and uniformly bounded functions and $(u_k)_k$ is a sequence of uniformly bounded viscosity solutions of,
$$
\left\{
\begin{array}{rcl}
|\nabla u_k + \vec q_k|^{\gamma}F(D^2 u_k) &=& f_k(x),  \quad \mbox{in} \, B^+_{1},\\
\partial_n u_k  &=& g_k  \quad \quad \,\,\,\, \mbox{in} \,\, \Upsilon
\end{array}
\right.
$$
Then the sequence $(u_k)_k$ is relatively compact in $C(\overline{B^+_1})$. In particular, if $u_k \to u_{\infty}$, $\vec q_k \to \vec q_{\infty}$ and $f_k \to f_{\infty}$ and $g_k \to g_{\infty}$, then $u_{\infty}$ is a viscosity solution to 
$$
\left\{
\begin{array}{rcl}
|\nabla u_{\infty} + \vec q_{\infty}|^{\gamma}F(D^2 u_{\infty}) &=& f_{\infty}(x),  \quad \mbox{in} \, B^+_{1},\\
\partial_n u_{\infty} \cdot \nu &=& g_{\infty}  \quad \quad \,\,\,\, \mbox{in} \,\, \Upsilon
\end{array}
\right.
$$
\end{corollary}

To end the section, as in \cite{Silv1}, we use the following lemma.
\begin{lemma} \label{Canc}
Let $\vec q \in \mathbb{R}^n$, $\varphi \in C(\partial B_1 \cap \{x_n>0\})$, $g \in L^{\infty}(\Upsilon)$ and $u$ be a viscosity solution of
$$
\left\{
\begin{array}{rcl}
|\nabla u + \vec q|^{\gamma}F(D^2 u) &=& 0   \quad \,\,\,\, \mbox{in} \,\,\,\, B^+_{1}\\
u &=& \varphi \quad \,\,\, \mbox{in} \,\,\,\, \partial B_1 \cap \{x_n>0\}\\
\partial_n u  &=& g   \quad \,\,\,\, \mbox{in} \,\,\,\, \Upsilon
\end{array}
\right.
$$
Then $u$ is a viscosity solution of
$$
\left\{
\begin{array}{rcl}
F(D^2 u) &=& 0   \quad \,\,\,\, \mbox{in} \,\,\,\, B^+_{1}\\
u &=& \varphi \quad \,\,\, \mbox{in} \,\,\,\, \partial B_1 \cap \{x_n>0\}\\
\partial_n u  &=& g   \quad \,\,\,\, \mbox{in} \,\,\,\, \Upsilon
\end{array}
\right.
$$
\end{lemma}
\begin{proof}
We can reduce the problem to the case $\vec q=0$ since $v = u+ \vec q \cdot x$ solves $|\nabla v|^{\gamma}F(D^2 v)=0$. It is sufficient to prove the super-solution property since the sub-solution property is similar.  Suppose that $P(x)$ touches $u$ from below (resp. above) at a point $x_0 \in \overline{B^+_1}$.   If $x_0 \in B^+_1$, then by Lemma 6 in \cite{Silv1}, $F(D^2 P(x_0)) \ge 0 \,\, \textrm{(resp.} \,\, \le 0 \textrm{)}$. If $x_0 \in \Upsilon$, we know that $P$ satises also the same inequality for the normal derivatives in the point $x_0 \in \Upsilon$.
\end{proof}

\section{Proof of Main Result}\label{SC2}

In this section we give a proof for Theorem \ref{P1}. We assume that $\gamma >0$ and $f \in L^{\infty}(B^+_1)$, $g \in C^{\beta}(\Upsilon)$, $\varphi \in C(\partial B_1 \cap \{x_n>0\})$, and we want to show that any viscosity solution $u$ of 
\begin{equation} \label{Eq2}
\left\{
\begin{array}{rcl}
|\nabla u|^{\gamma}F(D^2u) &=& f \quad \mbox{in}  B^+_1\\
u &=& \varphi \quad \mbox{in} \partial B_1 \cap \{x_n>0\}\\
\partial_n u &=& g \quad \mbox{in} \,\, \Upsilon .  
\end{array}
\right.
\end{equation}
is in $C^{1,\alpha}(\overline{B_{1/2}}\cap \{x_n=0\})$, where 
\begin{equation}\label{Hol1}
	\alpha \colon= \min\left\{\alpha^{-}_0, \beta, \frac{1}{1+\gamma}\right\}
\end{equation}
where $\alpha_0$ is the optimal exponent of regularity theory for homogeneous equations
$$
	F(D^2 u)=0 \quad \textrm{in} \quad B^+_1
$$
with Neumann boundary condition $\partial_n u  =0$ in $\Upsilon$(see Theorem 6.1 in \cite{Silv1}), where the estimate indicated in \eqref{Hol1} should be read as
$$
    \left\{
   \begin{array}{rcl}
   \textrm{If} \quad \min\left\{\beta,\frac{1}{1+\gamma}\right\}<\alpha_0,& \mbox{then} & u \in C^{1,\min\{\beta,\frac{1}{1+\gamma}\}}\\
   \textrm{If} \quad \min\left\{\beta,\frac{1}{1+\gamma}\right\} \ge\alpha_0,  & \mbox{then} & u \in C^{1,\sigma}, \,\, \textrm{for any} \,\,\, 0 < \sigma<\alpha_0
   \end{array}
   \right.
   $$
   \begin{remark} \label{Reduc}
At this point, by a standard  argument as in \cite{BB}, one can combine the boundary $C^{1,\alpha}$ estimate with the optimal interior regularity obtained in \cite{ART2} to conclude that $u \in C^{1,\alpha}(\overline{B^+_{1/2}})$. Now going back to the original domain $\Omega$, we can assert that there exists an affine approximation for $u$ of order $1+\alpha$ at all points of $\partial \Omega \cap B_1$.  Again, by a standard argument as in \cite{BB}, we concluded that $u \in C^{1,\alpha}(\overline{\Omega \cap B_{1/2}})$.
\end{remark}

 \subsection{Reduction of the problem}

In this section, we first show that a simple re-scaling reduces the proof of the problem to the case that $\|u\|_{L^{\infty}(\overline{B^+_1})} \le 1$, $\|g\|_{C^{\beta}(\Upsilon)} \le \epsilon_0$ and $\|f\|_{L^{\infty}(\overline{B^+_1})} \le \epsilon_0$ for some small constant $\epsilon_0$ which will be chosen later. We then further reduce the proof to an improvement of flatness lemma.

\begin{proposition}
In order to prove Theorem \ref{P1}, it is enough to prove that
$$
	\|u\|_{C^{1,\alpha}(\overline{B^+_{1/2}})} \le C
$$
assuming $\|u\|_{L^{\infty}(\overline{B^+_1})} \le 1$, $\|g\|_{C^{\beta}(\Upsilon)} \le \epsilon_0$ and $\|f\|_{L^{\infty}(\overline{B^+_1})} \le \epsilon_0$ for some $\epsilon_0 >0$ depends on the ellipticity constants, dimension and $\gamma$.
\end{proposition}
\begin{proof}
Given any function $u$ under the assumptions of Theorem \ref{P1}, we can take
$$
	\eta = \left[ \|u\|_{L^{\infty}(\overline{B^+_1})} + \epsilon^{-1}_0 \left( \|g\|_{C^{\beta}(\Upsilon)}+ \|f\|_{L^{\infty}(\overline{B^+_1})}\right) \right]^{-1}
	$$
and consider the scaled function $\tilde{u}(x) = \eta u(x)$ solving the equation
$$
		\left\{
\begin{array}{rcl}
 |\nabla \tilde{u}|^{\gamma}F_{\eta}(D^2 \tilde{u}) &=& \tilde{f},  \quad \mbox{in} \, B^+_1\\
\partial_n \tilde{u} &=& \tilde{g}  \quad \,\,\,\mbox{in} \,\, \Upsilon
\end{array}
\right.
$$
where $F_{\eta}(M) = \eta F(\eta^{-1} M)$, $\tilde{f}(x) = \eta^{1+\kappa}f(x)$ and $\tilde{g}(x) = \eta g(x)$.  Note that, the function $F_{\eta}(\cdot)$ has the same ellipticity constant as $F(M)$.  But now $\|\tilde{u}\|_{L^{\infty}(\overline{B^+_1})} \le 1$, $\|g\|_{C^{\beta}(\overline{B^+_1})} \le \epsilon_0$ and $\|\tilde{f}\|_{L^{\infty}(\overline{B^+_{1}})}\le \epsilon_0$. Therefore, if

$$
	\| \tilde{u} \|_{C^{1,\alpha}(\overline{B^{+}_{1/2}})} \le C,
$$
by scaling back to $u$, we get
$$
\|u\|_{C^{1,\alpha}(\overline{B^+_{1/2}})} \le C \left(\|u\|_{C(\overline{B^+_1})} + \|g\|_{C^{\beta}(\Upsilon)} + \|f\|_{L^{\infty}(\overline{B^+_1})}\right)
$$
\end{proof}
\subsection{A geometric tangential approach}

We proceed in a way similar to \cite{CC} and prove an approximation lemma with Neumann boundary condition first. Using this result we iterate the estimates of the theory to obtain $C^{1,\alpha}$-estimates   adapting Lemma 5.1 in \cite{ART2} to our case.

\begin{lemma}[Approximation Lemma]\label{PT1}
Let $\vec q \in \mathbb{R}^n$ be an arbitrary vector. Suppose that $\varphi \in C(\partial B_1 \cap \{x_n>0\})$, has $\rho=\rho(s)$ as modulus of continuity on $\partial B_1 \cap \{x_n>0\}$ and satisfies $\|\varphi\|_{L^{\infty}(\partial B_1 \cap \{x_n>0\})} \le K$, for some positive constant $K$. Then, given $\delta>0$, there exists $\epsilon_0>0$ depending only on $\delta, n,\lambda,\Lambda, \rho, K$ such that if 
$$
	\|f\|_{L^{\infty}(B^+_1)} \le \epsilon_0 \quad \textrm{and} \quad \|g\|_{L^{\infty}(\Upsilon)} \le \epsilon_0
$$
then any two viscosity solutions $v$ and $w$ of, respectively
$$
\left\{
\begin{array}{rcl}
|\nabla w + \vec q|^{\gamma}F(D^2 w) &=& f   \quad \,\, \mbox{in} \,\,\,\, B^+_{1}\\
w &=& \varphi \quad \,\, \mbox{on} \,\,\,\, \partial B_1 \cap \{x_n>0\}\\
\partial_n w &=& g   \quad \,\,\,\, \mbox{in} \,\,\,\, \Upsilon
\end{array}
\right.
 $$
 and
 $$ 
\left\{
\begin{array}{rcl}
F(D^2 h) &=& 0   \quad \,\, \mbox{in} \,\,\,\, B^+_{1}\\
h &=& \varphi \quad \,\, \mbox{on} \,\,\,\, \partial B_1 \cap \{x_n>0\}\\
\partial_n h &=& 0   \quad \,\,\,\, \mbox{in} \,\,\,\, \Upsilon
\end{array}
\right.
$$
satisfy
$$
	\|w-h\|_{L^{\infty}(B^+_1)} \le \delta.
$$
\end{lemma}
\begin{proof}
We argue by contradiction. Suppose that the lemma is not true. Then there exist $\delta_0>0$ and a sequence of operators $F_k$ and functions $h_k,w_k \in C(\partial B_1 \cap \{x_n>0\} ), \, g_k \in C^{\beta}(\Upsilon)$, $\vec q_k$ and $f_k$ (which satisfy the hypothesis for $F$, $\varphi$, $g$ and $f$, respectively, in the lemma) for which there are viscosity solutions $w_k$ and $h_k$ of
$$
\left\{
\begin{array}{rcl}
|\nabla w_k + \vec q_k|^{\gamma}F(D^2 w_k) &=& f _k  \quad \,\, \mbox{in} \,\,\,\, B^+_{1}\\
w_k &=& \varphi_k \quad \,\, \mbox{on} \,\,\,\, \partial B_1 \cap \{x_n>0\}\\
\partial_n w_k &=& g_k   \quad \,\,\,\, \mbox{in} \,\,\,\, \Upsilon
\end{array}
\right.
 $$
 and
 $$ 
\left\{
\begin{array}{rcl}
F(D^2 h_k) &=& 0   \quad \,\, \mbox{in} \,\,\,\, B^+_{1}\\
h_k &=& \varphi_k \quad \,\, \mbox{on} \,\,\,\, \partial B_1 \cap \{x_n>0\}\\
\partial_n h_k &=& 0   \quad \,\,\,\, \mbox{in} \,\,\,\, \Upsilon
\end{array}
\right.
$$
such that 
\begin{equation}\label{A1}
	\|f_k\|_{L^{\infty}(B^+_1)} \to 0, \,\, \|g_k\|_{L^{\infty}(\Upsilon)} \to 0 \quad \textrm{as} \,\,\, k \to \infty
\end{equation}
and
\begin{equation}\label{AA1}
	\|w_k-h_k\|_{L^{\infty}(B^+_1)} > \delta_0 \quad \forall \,\, k.
\end{equation}

Since all $\varphi_k$'s have the same modulus of continuity $\rho$ on $\partial B_1 \cap \{x_n>0\}$, $\|\varphi_k\|_{L^{\infty}(\partial B_1 \cap \{x_n>0\})} \le K$ for any $k$, we may suppose that $\varphi_k \to \varphi_{\infty}$ uniformly in $\partial B_1 \cap \{x_n>0\}$ and apply Lemma 9.1 in \cite{Silv}, conclude that $\{h_k\}$ is equicontinuous (and uniformly bounded) sequence of functions in $\overline{B^+_1}$.  Therefore, taking subsequence, we may assume that $h_k \to h_{\infty}$ \ uniformly in $\overline{B^+_1}$ and $h_{\infty}=\varphi_{\infty}$ in $\partial B_1 \cap \{x_n>0\}$. We may also assume (again by the Arzela-Ascoli theorem) that $F_{k}(\cdot) \to \mathfrak{F}_{\infty}(\cdot)$ as $k \to \infty$ uniformly in compact of $\textrm{Sym}(n)$ (the space of symmetric matrices), for some uniformly elliptic operator $\mathfrak{F}_{\infty}$. Thus, $h_{\infty} \in C(\overline{B^+_1})$ is a viscosity solution of 
 $$ 
\left\{
\begin{array}{rcl}
\mathfrak{F}_{\infty}(D^2 h_{\infty}) &=& 0   \quad \,\, \mbox{in} \,\,\,\, B^+_{1}\\
h_{\infty} &=& \varphi_{\infty} \quad \,\, \mbox{in} \,\,\,\, \partial B_1 \cap \{x_n>0\}\\
\partial_n h_{\infty} &=& 0   \quad \,\,\,\, \mbox{in} \,\,\,\, \Upsilon
\end{array}
\right.
$$

Our goal is to show that the sequence $w_k$ is pre-compact in the $C^0(\overline{B^+_{1}})$-topology.  We can obtain a universally large constant $A_0>0$, such that, given a subsequence $\{\vec q_{k}\}$, where
$$
	|\vec q_{k}| \ge A_0, \quad \forall k \in \mathbb{N},
$$
then, by Lemma \ref{L2},$\{w_{k}\}_{j\in \mathbb{N}}$ is bounded in $C^{0,1}(\overline{B^+_{1}})$.
For the case where,
$$
	|\vec q_{k}| < A_0, \quad \forall k \in \mathbb{N},
$$
easily follows (see Lemma \ref{L1}) that the sequence $\{w_{k}\}_{k\in \mathbb{N}}$ is bounded in $C^{0,\beta_1}(\overline{B^+_{1}})$ for some $0< \beta_1 < 1$.  Therefore in general, the family $\{w_k\}_{j\in\mathbb{N}}$ is bounded in $C^{0,\beta_1}(\overline{B^+_{1}})$, which gives us the desired compactness. From the compactness previously established,  up to a subsequence, $w_k \to w_\infty$ locally uniformly in $\overline{B^+_{1}}$.  Therefore, taking subsequences, we may assume that
$$
	w_k \to w_{\infty} \quad \textrm{and} \quad h_k \to h_{\infty} \quad \textrm{uniformly in} \, \overline{B^+_1} \quad \textrm{as} \,\,\, k \to \infty,
$$
for some $w_{\infty},h_{\infty} \in C(\overline{B^+_1})$ such that
$$
	w_{\infty} = h_{\infty} =\varphi_{\infty}\quad \textrm{in} \quad  \partial B^+_1 \cap \{x_n>0\}.
$$
Our ultimate goal is to prove that the limiting function $w_\infty$ is a solution to  a constant coefficient, homogeneous, $(\lambda, \Lambda)$-uniform elliptic equation. For that we also divide our analysis in two cases.
\begin{itemize}
\item If $|\vec q_j|$ bounded, we can extract a subsequence of $\{\vec q_j\}$, that converges to some $\vec q_\infty\in \mathbb{R}^n$. Finally, by Corollary \ref{A1}, we conclude
$$
\left\{
\begin{array}{rcl}
|\nabla w_{\infty} + \vec q_\infty|^{\kappa} \mathfrak{F}_\infty(D^2w_\infty) &=& 0,  \quad \mbox{in} \, B^+_{1},\\
w_{\infty} &=& \varphi_{\infty} \quad \mbox{in} \quad \partial B_{1} \cap \{x_n>0\}\\
\partial_n w_{\infty}&=& 0  \quad \mbox{in} \,\, \Upsilon
\end{array}
\right.
$$
For some $\mathfrak{F}_\infty$ $(\lambda,\Lambda)$-elliptic. So, by Lemma \ref{Canc}, we conclude that $w_\infty$ also satisfies
$$
\left\{
\begin{array}{rcl}
 \mathfrak{F}_\infty(D^2w_\infty) &=& 0,  \quad \mbox{in} \, B^+_{1},\\
 w_{\infty} &=& \varphi_{\infty} \quad \mbox{in} \quad \partial B_{1} \cap \{x_n>0\}\\
\partial_n w_{\infty} &=& 0  \quad \mbox{in} \,\, \Upsilon
\end{array}
\right.
$$
in the viscosity sense.

\item If $|\vec q_j|$ unbounded, then without loss of generality $|\vec q_j| \to \infty$. In this case, define $\vec{e}_k = \vec q_k/|\vec q_k|$ and so $w_k$ satisfies
$$
\left\{
\begin{array}{rcl}
\left| \vec{e}_{j}+\frac{\nabla w_{k}}{|\vec q_{j}|}\right| F_{k}(D^2w_{k})&=&\frac{f_{k}}{|\vec q_{k}|^\gamma} \quad \mbox{in} \, B^+_{1},\\
w_{k} &=& \varphi_k \quad \mbox{in} \quad \partial B_{1} \cap \{x_n>0\}\\
 \partial_n w_k&=& g_k \quad \quad \mbox{in} \,\, \Upsilon.
\end{array}
\right.
$$
We get at the limit
$$
\left\{
\begin{array}{rcl}
 |\vec{e}_{\infty} + 0\cdot \nabla w_{\infty}|\mathfrak{F}_{\infty}(D^2w_{\infty})&=& 0 \quad \quad \mbox{in} \, B^+_{1},\\
 w_{\infty} &=& \varphi_{\infty} \quad \quad \mbox{in} \quad \partial B_{1} \cap \{x_n>0\}\\
 \partial_n w_{\infty} &=& 0  \quad \quad \mbox{in} \,\, \Upsilon
\end{array}
\right.
$$
for $|\vec{e}_{\infty}|=1$. So, $w_\infty$ satisfies
\begin{equation} \label{AA0}
\left\{
\begin{array}{rcl}
 \mathfrak{F}_{\infty}(D^2w_{\infty})&=& 0 \quad \quad \mbox{in} \, B^+_{1},\\
 w_{\infty} &=& \varphi_{\infty}  \quad \mbox{in} \quad \partial B_{1} \cap \{x_n>0\}\\
\partial_n w_{\infty}&=& 0  \quad \quad \mbox{in} \,\, \Upsilon
\end{array}
\right.
\end{equation}
in the viscosity sense. Since \eqref{AA0} is uniquely solvable we get $v_{\infty} = w_{\infty}$ in $\overline{B^+_1}$, which also yields a contradiction to $\eqref{AA1}$
\end{itemize}

\end{proof}

\subsection{Universal flatness improvement}

In this Section, we deliver the core sharp oscillation decay that will ultimately imply the optimal $C^{1,\alpha}$ regularity estimate for solutions to Eq. \eqref{E1}. The first task is a step-one discrete version of the aimed optimal regularity estimate. A quick inference on the structure of equation \eqref{E1} reveals that no universal regularity theory for such equation could go beyond $C^{1,\alpha_0}$,  where $\alpha_0$ denote the optimal H\"{o}lder continuity exponent for solutions constant coefficients, homogeneous, elliptic equation 
$$
	F(D^2 h)=0 \quad \textrm{in} \quad B^+_1
$$
with $\nabla h \cdot \nu = 0$ on $\Upsilon$.  In fact the degeneracy term forces solutions to be less regular than solutions to the uniformly elliptic problem near its set of critical points. In this present section we show that a viscosity solution, $u$, to \eqref{E1}
is pointwise differentiable and its gradient, $\nabla u$, is locally of class $C^{0,\min\{\alpha^-_0, \beta,\frac{1}{1+\gamma}\}}$, which is precisely the optimal regularity for degenerate equations of the type \eqref{E1}.  This is the contents of next Lemma.
\begin{lemma} \label{L1}
Let $\vec q \in \mathbb{R}^n$ be an arbitrary vector and suppose that $\varphi \in C(\partial B_1 \cap \{x_n>0\})$. Let $w$ a normalized, i.e., $\|w\|_{L^{\infty}(\overline{B^+_1})} \le 1$, viscosity solution to
$$
 \left\{
\begin{array}{rcl}
 |\nabla w + \vec q|^{\gamma} F(D^2w) &=& f  \,\, \quad \mbox{in} \, B^+_1\\
 w &=& \varphi \quad \mbox{on} \,\, \partial B_1 \cap \{x_n>0\}\\
 \partial_n w &=& g  \,\, \quad \mbox{in} \,\, \Upsilon
\end{array}
\right.
 $$
 Given $\alpha \in (0,\alpha_0) \cap (0,\frac{1}{1+\gamma} ]$, there exists $0 < \rho_0 < 1/2$ and $\epsilon_0 >0$, depending only upon $n,\lambda,\Lambda, \gamma$ and $\alpha$, such that if
 $$
 	\|f\|_{L^{\infty}(B^+_1)} + \|g\|_{C^{\beta}(\Upsilon)} \le \epsilon_0,
 $$
 then there exists an affine function $\ell(X) = a + \vec b \cdot X$ such that
  \begin{eqnarray*}
  \|w - \ell\|_{L^{\infty}(B^+_{\rho_0})}  &\le& \rho^{1+\alpha}_0\\
	\langle \vec b , \vec e_n \rangle   &=& 0.
 \end{eqnarray*}
 Furthermore,
 $$
 	|a| + |\vec b| \le C(n,\lambda,\Lambda).
 $$
\end{lemma}
\begin{proof}
For a $\delta>0$ to be chosen a posteriori, let h be a solution to a constant coefficient, homogeneous, $(\lambda,\Lambda)$-uniform elliptic equation that is $\delta$-close to $v$ in $L^{\infty}(\overline{B^+_{1}})$. The existence of such a function is the thesis of Lemma \ref{PT1}, provided $\epsilon_0$ is chosen small enough, depending only on $\delta$ and universal parameters. Since our choice for $\delta$ later in the proof will depend only upon universal parameters, we will conclude that the choice of $\epsilon_0$ is also universal. 
From normalization of $v$, it follows that $\|h\|_{L^{\infty}(\overline{B^+_{1}})}\le 2$. Therefore, from the regularity theory available for $w$, see for instance \cite{Silv}, Theorem 9.3 , we can estimate
\begin{eqnarray}
	 \|h(X) - \left(\nabla h(0) \cdot X + h(0)\right)\|_{L^{\infty}(B^+_{1/2})} &\le& C(n,\lambda,\Lambda,, \alpha) \cdot r^{1+\alpha_0} \label{E2}\\
	|\nabla h(0)| + |h(0)| &\le& C(n,\lambda,\Lambda,\alpha_0) \nonumber 
\end{eqnarray}
and by boundary condition  $\langle \nabla h(0) , \vec e_n \rangle= 0$. Remark that, by Theorem 6.1 in \cite{Silv}, $\nabla h$ is well defined up to the boundary $\Upsilon$. Let us label
$$
	\ell(X) = \nabla h(0) \cdot X + h(0).
$$
It readily follows from triangular inequality that
\begin{equation} \label{E3}
\|w(x) - \ell(x)\|_{L^{\infty}(B^+_{\rho_0})} \le \delta + C(n,\lambda,\Lambda) \cdot \rho^{1+\alpha_0}_0
\end{equation}
Now, fixed an exponent $\alpha < \alpha_0$, we select $\rho_0$ and $\delta$ as
\begin{eqnarray}
	\rho_0 &=& \sqrt[\alpha_0- \alpha]{\frac{1}{2C(n,\lambda,\Lambda)}} \label{E4}\\
	\delta &=& \frac{1}{2} \left(\frac{1}{2 C(n,\lambda,\Lambda)}\right)^{\frac{1+\alpha}{\alpha_0-\alpha}} \label{E5}
\end{eqnarray}
where $C$ is the universal constant appearing in \eqref{E2}. We highlight that the above choices depend only upon $n,\lambda,\Lambda$ and the fixed exponent $0 < \alpha < \alpha_0$. Finally, combining \eqref{E2}, \eqref{E3}, \eqref{E4} and \eqref{E5}, we obtain
\begin{eqnarray*}
	 \|w(X)-\ell(X)\|_{L^{\infty}(B^+_{\rho_0})} &\le&  \frac{1}{2} \left(\frac{1}{2 C(n,\lambda,\Lambda)}\right)^{\frac{1+\alpha_0}{\alpha-\alpha_0}} + C(n,\lambda,\Lambda) \cdot \rho^{1+\alpha}_0 \\
	&=& \frac{1}{2} \rho^{1+\alpha}_0 + \frac{1}{2} \rho^{1+\alpha}_0 = \rho^{1+\alpha}_0,
\end{eqnarray*}
and the Lemma is proven.
 \end{proof}
Our next step involves iterating Lemma \ref{L1} in the appropriate geometric scaling.
\begin{lemma} \label{L2}
Consider $\varphi \in C(\partial B_1 \cap \{x_n>0\})$, $ g \in C^{\beta}(\Upsilon)$ for some $\beta \in (0,1)$, $f \in L^{\infty}(B^+_1)$ . Let $u$ be a viscosity solution to
$$
 \left\{
\begin{array}{rcl}
 |\nabla u|^{\gamma} F(D^2u) &=& f  \quad \mbox{in} \, B^+_1 \\
  u&=& \varphi \quad \mbox{on} \,\, \partial B_1 \cap \{x_n>0\}\\
 \nabla u \cdot \nu  &=& g  \quad \mbox{in} \,\, \Upsilon,
\end{array}
\right.
 $$
with $\|u\|_{L^{\infty}(\overline{B^+_1})}\le 1$ and $\alpha = \min\left\{\alpha_0, \beta, \frac{1}{1+\gamma}\right\} $. Then, there exists $0<\rho_0 < 1/2$ and $\epsilon_0 \in [0,1]$ only depending on $\lambda, \Lambda$, $n$ and $\gamma$ such that, if
$$
	\|f\|_{L^{\infty}(B^+_1)} + \|g\|_{C^{\beta}(\Upsilon)} \le \epsilon_0
$$
 then for all $j \in \mathbb{N}$, there exists a sequence of affine functions $\ell_j(x) = a_j + \vec b_j \cdot x$ satisfying
\begin{eqnarray}
	|a_{j+1}-a_{k}| + \rho^j_0 |\vec b_{j+1}-\vec b_j| &\le& C_0 \rho^{(1+\alpha)j}_{0} \label{6.8}\\
	\langle \vec b_j , \vec e_n \rangle& =&g(0), \label{6.9}
\end{eqnarray}
such that
\begin{equation} \label{E6}
	 \|u(x) - \ell_j(x) \|_{L^{\infty}(B^+_{\rho^j_0})} \le \rho^{j(1+\alpha)}_0.
\end{equation}
\end{lemma}
\begin{proof}
 
We argue by finite induction. The case $k=1$ is precisely the statement of Lemma \ref{L1}. Suppose we have verified  \eqref{E6} for $k=1,2,\ldots, j$. Define the rescaled function $w : \overline{B^+_1}\rightarrow \mathbb{R}$
$$
	w(x) \colon= \frac{(u-\ell_j)(\rho^j_0x)}{\rho^{j(1+\alpha)}_0}
$$
It readily follows from the induction assumption that $\|w\|_{L^{\infty}(\overline{B^+_1})} \le 1$. Furthermore, $w$ satisfies
$$
 \left\{
\begin{array}{rcl}
 |\nabla w + \rho^{-\alpha j}_0 \vec b_j|^{\gamma} F_{j}(D^2 w) &=& \tilde{f}  \quad \mbox{in} \, B^+_{1}\\
  w &=& \tilde{\varphi} \quad \mbox{on} \,\, \partial B_1 \cap \{x_n>0\}\\
\partial_n w &=&\tilde{ g} \quad \mbox{in} \,\, \Upsilon.
\end{array}
\right.
 $$
 where
 \begin{eqnarray*}
 	\tilde{f}(x) &=& \rho^{j[1-\alpha(1+\gamma)]}f(\rho^j_0 X)\\
	\tilde{\varphi} &=& \rho^{-j(1+\alpha)}_{0} \left(\varphi(\rho^{j}_0 x) - \ell_{j}(\rho^{j}_{0} x)\right)\\
	\tilde{g}(x) &=& \rho^{-\alpha j}_0g(\rho^j_0 x)
 \end{eqnarray*}
 and
 $$
 	F_{j}(M) \colon= \rho^{j(1-\alpha)}_{0} F \left(\frac{1}{\rho^{j(1-\alpha)}_0} M\right)
 $$
  It is standard to verify that the operator $F_{j}$ is $(\lambda,\Lambda)$-elliptic. Note that, by induction hypothesis $\|\tilde{\varphi}\|_{L^{\infty}(\partial B_1 \cap \{x_n>0\})} \le 1$.  Also, one easily estimate
 \begin{eqnarray}
 	\|f_{j}\|_{L^{\infty}(\overline{B^+_1})}& \le& \rho^{j [1-\alpha(1+\gamma)]}_{0} \|f\|_{L^{\infty}(\overline{B_{\rho^k_0}})} \label{E7}\\
	\|\tilde{g}\|_{C^{\beta}(\Upsilon)}&\le& \rho^{{(\beta-\alpha)j}}_{0} \|g\|_{C^{\beta}(\Upsilon \cap B_{\rho^j_0})}. \label{E8}
 \end{eqnarray}
 Due to the sharpness of the exponent selection $\alpha = \min \left\{\alpha_0, \beta, \frac{1}{1+\gamma}\right\}$, namely $\alpha \le \frac{1}{1+\gamma}$ and $\alpha \le \beta$, we conclude $(F_j,f_j,g_j)$ satisfies the smallness assumption Lemma \ref{L1}. Thus, there exists a affine function $\tilde{\ell}(X) \colon= a + \vec b \cdot X$ with
 \begin{eqnarray}
 	|a| + |\vec b| &\le& C(n,\lambda,\Lambda) \label{E11-1}\\
	\nonumber\\
	\langle  \vec b , \vec e_n \rangle &=&0, \label{E11}
 \end{eqnarray}
 such that
 \begin{equation} \label{E10}
 		 \|w(X)-\tilde{\ell}(X)\|_{L^{\infty}(B^+_{\rho_0})} \le \rho^{1+\alpha}_0.
 \end{equation}
 In the sequel, we define the $(j+1)$th approximating affine function,
 $$
 	\ell_{j+1}(X) \colon= a_{j+1} + \vec b_{j+1} \cdot X,
 $$
 where the coefficients are given by
 $$
 	a_{j+1} \colon= a_k + \rho^{j(1+\alpha)}_0 a \quad \textrm{and} \quad \vec b_{j+1} \colon= \vec b_j + \rho^{\alpha j}_0 \vec b.
 $$
 By induction assumption and \eqref{E11}, $b_{j+1} \cdot \nu  =g(0)$. Re-scaling estimate \eqref{E10} back, we obtain
 $$
 	 \|u(x) - \ell_{j+1}(x)\|_{L^{\infty}(B^+_{\rho^{j+1}_0})} \le \rho^{(j+1)(1+\alpha)}_0
 $$
 and the proof of Lemma is complete.
\end{proof}

\subsection{Proof of Theorem \ref{P1}}

We now conclude the proof of Theorem \ref{P1}.  In view of our discussion in remark \ref{Reduc}, by optimal interior estimates (see \cite{ART2}) it is enough to find a $C^{1,\alpha}$ estimate for the points in $\overline{B_{1/2}} \cap \{x_n=0\}$.  Moreover, for proving \eqref{flatq} (flat boundary) it is enough to get a universal estimate at the origin, and then apply it to rescaling and translations of $u$ ( under the hypotheses
of Lemmas \ref{L1} and \ref{L2}). In fact, letting $v=u-g(0)x_n -u(0)$, we have that $v(0)=0$ and if $\ell_{\star}$ is the affine approximation of order $1+\alpha$ at $0$ for $v$, then $\ell_{\star} +g(0)x_n$ is the affine approximation for $u$ at $0$.   For a fixed exponent $\alpha$ satisfying the sharp condition, we will
establish the existence of an affine function
$$
	\ell_{\star}(X) \colon= a_{\star} + \vec b_{\star} \cdot x,
$$
such that
$$
	|a_{\star}| + |\vec b_{\star}| \le C,
$$
and
\begin{eqnarray*}
	\sup_{x \in B^+_r} |u(x) - \ell_{\star}(x)| &\le& C \cdot r^{1+\alpha}, \quad \forall \, r \ll 1,
\end{eqnarray*}
for a constant $C$ that depends only on $n,\lambda, \Lambda, \kappa$ and $\alpha$. Initially, we notice that it follows from \eqref{6.8} that the coefficients of the sequence of affine functions $\ell_k$ generated in Lemma \ref{L2}, namely $\vec b_j$ and $a_j$, are Cauchy sequences in $\mathbb{R}^n$ and in $\mathbb{R}$, respectively. Let
\begin{eqnarray}
	\vec b_{\star} &\colon=& \lim_{j \to \infty} \vec b_j \label{7.4}\\
	a_{\star} &\colon=& \lim_{j \to \infty} a_{j} \label{7.5}
\end{eqnarray}
It also follows from the estimate obtained in \eqref{6.8} that
\begin{eqnarray*}
	|a_{\star} - a_{j}| &\le& \frac{C_0}{1-\rho_0} \rho^{j(1+\alpha)}_0, \label{7.6}\\
	|\vec b_{\star} - \vec b_j| &\le& \frac{C_0}{1-\rho_0} \rho^{j \alpha}_0 \label{7.7}.
\end{eqnarray*}
Now, fixed a $0 < r < \rho_0$, we choose $j \in \mathbb{N}$ such that
$$
	\rho^{j+1}_0 < r \le \rho^{j}_0.
$$
We estimate
\begin{eqnarray*}
	\sup_{x \in B^+_r} |u(x) - \ell_{\star}(x)| &\le& \sup_{x \in B^+_{\rho^j_0}} |u(x) - \ell_{\star}(x)|\\
	&\le&  \sup_{x \in B^+_{\rho^j_0}} |u(x) - \ell_{j}(x)| +  \sup_{x \in B^+_{\rho^j_0}} |\ell_k(x) - \ell_{\star}(x)|\\
	&\le& \rho^{j(1+\alpha)}_0 + \frac{C_0}{1-\rho_0} \rho^{j(1+\alpha)}_0\\
	&\le& \frac{1}{\rho^{1+\alpha}_0} \left[1 + \frac{C_0}{1-\rho_0}\right] \cdot r^{1+\alpha},
\end{eqnarray*}
and the proof of Theorem is finally complete.

\section{Some consequences of the main result} \label{Four}

In this Section, we will present the some consequences of the main result. An important consequence is the following:
\begin{corollary}\label{Cor1}
Let $u$ be a viscosity solutions to 
$$
	|\nabla u|^{\gamma} F(D^2 u) =f \quad B^+_1
$$
with Neumann boundary condition $\nabla u \cdot \nu =g$ in $\Upsilon$. Assume $f \in L^{\infty}(B^+_1)$, $g \in C^{\beta}(\Upsilon)$ and $F$ is uniformly elliptic and concave. Then $u \in C^{1, \min\{\beta,\frac{1}{1+\gamma}\}}(\overline{B^+_{1/2}})$ and this regularity is optimal.  
\end{corollary} 
\begin{proof}
Corollary follows from Theorem \ref{P1} since solutions to concave equations with Newmann boundary conditions are of class $C^{1,1}$ by \cite{Silv1}.
\end{proof}

Another interesting example to visit is the $\infty$-Laplacian operator 
$$
	\Delta_{\infty} v \colon= (\nabla v)^{t} \cdot D^2 v \cdot \nabla v.
$$
The theory of infinity-harmonic functions, i.e., $\Delta_{\infty} v=0$, has received great deal of attention. One of the main open problems in the modern theory of PDEs is whether infinity-harmomic functions are of class $C^1$. This conjecture has been answered positively by Savin \cite{Savin}  in the plane. Over here, we would like to mention that although the conjecture is open, nevertheless it is well known that that solutions to $\Delta_{\infty} v=0$ are locally of class $C^{1,\alpha}$ in the plane, for some exponent $\alpha$ depending only $n$ for instance, \cite{Evans} and  quite recently, Evans and Smart \cite{Evans2} proved that infinity-harmonic functions are everywhere differentiable regardless the dimension. We remember the famous example of the infinity-harmonic function $u(x,y)=x^{4/3}-y^{4/3}$ due to Aronsson from the late 1960s sets the ideal optimal regularity theory for such problem. In high dimensions, the situation is quite different. Very recently, the $C^{1,1/3}$ conjecture has been solved in the context for obstacle problems by Rossi, J. , Teixeira, E. and Urbano M. in \cite{TU}.  

Let $h \in C(B_1)$ be an infinity-harmonic function. For each $p \gg1$, let $h_p$ be the solution to the boundary value problem
$$
 \left\{
\begin{array}{rcl}
 \Delta_{p} h_p &=& 0  \quad \mbox{in} \, B_{3/4} \\
  h_p&=& h \quad \mbox{on} \,\, \partial B_{3/4} 
\end{array}
\right.,
 $$
 where 
 $$
 	\Delta_p v \colon= |\nabla v|^{p-2} \Delta v + (p-2)|\nabla v|^{p-4} \Delta_{\infty} v
 $$
  is the $p$-Laplacian operator.  It is known that $h_p \to h$  uniformly to $h$. In particular 
  $$
  	\Delta_{\infty} h_p = o(1), \quad \textrm{as} \,\,\,  p \to \infty.
  $$ 
 Hereafter, let us call $h_p$ the $p$-harmonic approximation of the infinity-harmonic function $h$. Our contribution in the context of the conjecture is the following

\begin{theorem}
Let $f \in L^{\infty}(B^+_1)$, $g \in C^{\beta}(\Upsilon)$ and $h \in C(\overline{B^+_1})$ satisfy
$$
 \left\{
\begin{array}{rcl}
 |\Delta_{\infty} h_p| &=& O(p^{-1}) \quad \textrm{as} \,\,\, p \to \infty  \quad \mbox{in} \, \,\,B^+_1 \\
  u&=& \varphi \quad \mbox{on} \,\, \partial B_1 \cap \{x_n>0\}\\
 \nabla u \cdot \nu  &=& g  \quad \mbox{in} \,\, \Upsilon,
\end{array}
\right.
 $$
 in the viscosity sense.  Assume $u$ is smooth up to a possible radial singularity. Then $u \in C^{1, \min\{\beta, \frac{1}{3}\}}(\overline{B^+_{1/2}})$.
\end{theorem}

\begin{proof}
In fact, since $h_p$ is $p$-harmonic, it satisfies
$$
	|\nabla h_p|^{2} \Delta h_p = (2-p) \Delta_{\infty} h_p.
$$
From the Corollary \ref{Cor1}, we deduce $\|h_p\|_{C^{1,\min\{\beta,1/3\}}} \le C$.
\end{proof}

\section*{Acknowledgments}

\hspace{0.65cm} This work has been supported by Conselho Nacional de desenvolvimento cient\'{i}fico e Tecnol\'{o}gico  (CNPq-Brazil). GCR would like to thank Research Group on Partial Differential Equations, Dept. of Math.

\vspace{1cm}
\noindent  \textsc{Gleydson C. Ricarte} \hfill  \\
\hfill  Universidade Federal Cear\'a  \\
 \hfill Department of Mathematics \\
\hfill Fortaleza, CE-Brazil 60455-760\\
 \hfill \texttt{ricarte@ufc.br}

\end{document}